\documentclass{amsart}
                        \input xypic

\usepackage{amsthm, amsfonts, amssymb, amsmath, latexsym, enumerate, times}
\usepackage[all]{xy}
\usepackage[latin1]{inputenc}
\usepackage{mathrsfs}
\usepackage{array}
\usepackage{comment}
\usepackage{paralist}
\usepackage{color}

\newtheorem{theorem}{Theorem}
\newtheorem{lemma}[theorem]{Lemma}

\newtheorem{proposition}[theorem]{Proposition}

\theoremstyle{definition}

\newcommand{\Pic}{{\rm Pic}}

\renewcommand{\O}{{\mathcal O}}
\renewcommand{\L}{{\mathcal L}}
\newcommand{\R}{{\mathbb R}}
\newcommand{\bbZ}{{\mathbb Z}}
\newcommand{\bbF}{{\mathbb F}}

\renewcommand{\L}{{\mathcal L}}

\newcommand{\Xx}{{\mathcal X}}

\newcommand{\Proj}{{\mathbb P}}

\def\geq{\geqslant}
\def\leq{\leqslant}

\begin{document}
 
\title[Examples of non--effective rays]{Examples of non--effective rays at the boundary of the Mori cone of blow--ups of the plane}

\author{Ciro Ciliberto}
\address{Dipartimento di Matematica, Universit\`a di Roma Tor Vergata, Via O. Raimondo
 00173 Roma, Italia}
\email{cilibert@axp.mat.uniroma2.it}

\author{Rick Miranda}
\address{Department of Mathematics, Colorado State University, Fort Collins (CO), 80523,USA}
\email{rick.miranda@colostate.edu}
 
 
 
\maketitle

\centerline{We dedicate this paper to Giorgio Ottaviani on his 60th birthday}

\begin{abstract} In this paper we prove that no multiple of the linear system of plane curves of degree $d\geq 4$ with a point of multiplicity $d-m$ (with $2 \leq m \leq d$) and $m(2d-m)$ simple general points is effective.
\end{abstract}

\section*{Introduction} 

Alex Massarenti and Massimiliano Mella asked us the following question. Consider 13 general points $p_0,\ldots, p_{12}$ in the projective plane and consider the class of a quartic curve with a singular point at $p_0$ and passing through $p_1,\ldots, p_{12}$. Is it the case that no multiple of this class is effective? 

In trying to answer this question we  succeeded in proving a more general result. 
Namely, let $d$ be any integer and $p_0,\ldots, p_{m(2d-m)}$ general points in the plane with $m \leq d$. Consider the class (or system) $\xi_{d,m}$ of plane curves of degree $d$ with a point of multiplicity at least $d-m$ at $p_0$ and passing through $p_1,\ldots, p_{m(2d-m)}$. Then we have:

\begin{theorem}\label{thm:main} Fix $k \geq 1$.\\
\begin{inparaenum}
\item [(a)] For any $d \geq 4$ and any $m$ with $2 \leq m \leq d$, the class $k\xi_{d,m}$ is not effective.\\
\item [(b)] For any $d$, the system $\xi_{d,1}$ is a pencil of rational curves 
and $\xi_{d,0}$ is composed with the pencil of lines through the point $p_0$ 
and has dimension $d$. 
The multiple linear systems $k\xi_{d,1}$ are composed with the corresponding pencil 
and have dimension $k$. 
There is no member of these systems that contains an irreducible curve 
which is not a component of a member of this pencil. 
The same is true for the system $\xi_{2,2}$.\\
\item [(c)] For $d=3$, the systems $\xi_{3,3}$ and $\xi_{3,2}$ 
coincide with the system of cubics through $9$ general points, 
which consists of a unique cubic $C$. 
The systems $k\xi_{3,3}$ and $k\xi_{3,2}$ 
consist of the unique  divisor  $kC$.
\end{inparaenum}

\end{theorem}

A few remarks are in order. First statements (b) and (c) in the theorem are trivial (we stated them for completeness and we will inductively use them in the proof of (a) which is the core of the theorem).
Secondly, if $m=d$ or $m=d-1$ the statement is Nagata's theorem for $d^2$ general points in the plane (see \cite {N}), hence the theorem can be viewed as a generalization of Nagata's theorem.  So for the proof of (a) we may and will assume $2 \leq m \leq d-2$.

Let $X_n$ be the blow--up of the projective plane at $n$ general points. Let $\L_d(m_1,\ldots, m_n)$ be the linear system on $X_n$ corresponding to plane curves of degree $d$ with general points of multiplicities at least $m_1,\ldots, m_n$. 

If we blow up $p_0,\ldots, p_{m(2d-m)}$ we get the surface $X_{m(2d-m)+1}$ and $\xi_{d,m}$ can be interpreted as an element in $\Pic(X_{m(2d-m)+1})$; note that $\xi_{d,m}^2=0$. Moreover $\xi_{d,m}$ is nef. Indeed we consider a general plane curve $C$ of degree $d$ with a point $p_0$ of multiplicity $d-m$ and we can fix $m(2d-m)$ general points on $C$. If we blow up $p_0$ and the $m(2d-m)$ chosen points, the proper transform of $C$ is an irreducible curve with $0$ self--intersection, and therefore it is nef on the blow--up. Since nefness is an open condition, this is true for the general class $\xi_{d,m}$. 

Our result says that there is no positive number $k$ such that $\L_{kd}(k(d-m),k^{m(2d-m)})$ is non--empty (the exponential notation for repeated multiplicities is clear). If we set $N_1(X_{m(2d-m)+1})=\Pic(X_{m(2d-m)+1})\otimes _\bbZ \R$, then $\xi_{d,m}$ generates a rational ray in $N_1(X_{m(2d-m)+1})$ that is not effective (see \cite [\S 3.1] {CHMR}) and therefore it
sits in the boundary of the Mori cone of $X_{m(2d-m)+1}$. Such a ray, if rational in $N_1(X_{m(2d-m)+1})$, is called a \emph{good ray} in \cite [\S 3.2] {CHMR} whereas, if irrational, it is called a \emph{wonderful ray}. So far no wonderful ray has been discovered. However, proving that a given ray is good is in general not easy, and in \cite {CHMR} the authors were able to exhibit some examples. Therefore it is interesting to find good rays, and in this paper we make a new contribution in this direction.

Our proof uses the degeneration technique we introduced for analyzing the dimension of such linear system  (see, e.g., \cite {CM1}).  We briefly recall this in \S \ref {sec:deg}. The proof is by induction on $m$, the case $m=2$ being the critical one. We prove the $m=2$ case in \S \ref {sec:proof}. 
This particular example relies on a subtlety that requires us to analyze, 
more deeply than what we did in \cite {CM2}, 
the case in which there are multiple $(-1)$--curves splitting off a linear system in the limit. This we describe in \S \ref {sec:nthrow}. 
Finally in \S \ref {sec:proof} we finish the proof of Theorem \ref{thm:main}.

We notice that the surprising phenomenon that allows us to make the final analysis of the limit linear systems in the case $m=2$ is that we eventually end up with curves of a certain degree $tn$ in the plane with $n^2$ points of multiplicity $t$, which is currently the only case in which Nagata's conjecture is proven (see \cite {N}). Indeed we use here an argument inspired by the original one of Nagata (see l.c.) to deal with these cases. 

The ideas in this note can be generalized to prove more general similar results about general linear systems with zero self--intersection and we will do this in a forthcoming paper. 

\medskip

{\bf Acknowledgements:} Ciro Ciliberto is a member of GNSAGA of INdAM.

\section{The degeneration method}\label{sec:deg}

In this section we briefly recall the degeneration technique 
that we use to analyse planar linear systems (see \cite {CM1}). 
We want to study a linear system $\L_d(m_1,\ldots, m_n)$. 
To do this we consider a trivial family 
$\Proj^2\times \mathbb D\to \mathbb D$ over a disc $\mathbb D$. 
In the central fibre over $0\in \mathbb D$  we blow up a line $R$ 
producing a new family $\mathcal X\to \mathbb D$ 
with an exceptional divisor $F\cong \bbF_1$ 
and the proper transform $P\cong \Proj^2$ of the original central fibre. 
The new central fibre consists now of $F\cup P$, 
with $F,P$ transversely intersecting along the line $R$, 
which is the $(-1)$--curve in $F$. 

Next we fix $a$ general points on $P$ and $b$ general points of $F$, so that $a+b=n$. Consider sections of the family $\mathcal X\to \mathbb D$ 
extending these $n$ points to general points on the general fibre. 
Blowing up these sections, we have a degeneration of $X_n$ to 
the union of an $X_a$ (the blow--up of $P$ at the $a$ general points) 
and of an $X_{b+1}$ (the blow--up of $F$  at the $b$ general points). 

 Since there is an obvious map $\pi: \mathcal X\to \Proj^2$, 
we have the bundle $\O_\Xx(d)=\pi^*(\O_{\Proj^2}(d))$. 
This bundle restricts to the general fibre to $\O_{\Proj^2}(d)$. 
On the central fibre it restricts to the bundle $\O_{\Proj^2}(d)$ on $P$ and to $\O_F(df)$, where $f$ is the class of a fibre of the ruling of $F$ over $\Proj^1$.
This is a limit of the line bundle on the general fibre;
 there are other limits obtained by \emph{twisting} by  $\O_\Xx(-lP)$,
i.e., by tensoring the above limit bundle by $\O_\Xx(-lP)$, with $l$ an integer. 
This restricts to $\O_{\Proj^2}(d+l)$ on $P$
and to  $\O_F(df-lR)$\ on $F$. 
So we have a discrete set of limits of $\L_d(m_1,\ldots, m_n)$, 
depending on all choices for $a,b,l$ 
and distribution of the multiplicities among the $a+b$ points on the central fibre. 
 
A section of a limit line bundle is given by a pair of sections on $P$ and $F$, 
that restrict equally to the double curve $R$. 
We will call this the \emph{naive matching condition}. 
Such a section could be identically zero on one of the components of the central fibre, 
and in this case a matching section on the other component 
corresponds to a section of the linear system (called a \emph{kernel linear system}) 
of curves on the other component containing the double curve $R$.
One way to prove emptiness of the system on the general fibre 
is to find $a,b$, a distribution of the multiplicities and a twisting parameter $l$ 
such that there is no section of the limit line bundle on the central fibre 
that verifies the naive matching condition 
and that is non--zero on at least one of the two components.
 
An alternative approach to proving the emptiness of the linear system on the general fibre is the following. 
Suppose that  the system is non--empty on the general fibre. 
Then for every choice of $a,b$ and a distribution of the multiplicities, 
there will be a limit curve
 which must be the zero of a section of a limit line bundle 
given by some particular twist parameter $l$, 
this section being not identically zero on both $P$ and $F$.  
As we said, the naive matching condition means that the two curves restrict equally to $R$. However we will see in the next section that when the curves are non--reduced 
the matching conditions are more subtle. 
We will call these conditions \emph{refined matching conditions}.
Hence, to prove that the system on the general fibre is empty, it suffices to find $a,b$ and a distribution of the multiplicities so that for no $l$ there is a limit curve as above, i.e., a pair of curves on $P$ and $F$ satisfying the refined matching. This will be the approach we will use in the proof of the case $m=2$.

Clearly the former approach is easier than the latter;
however the latter approach could be necessary
if the naive approach fails for every twist, which can happen.

\section{Refined matching conditions}\label{sec:nthrow}

In this section we will perform an analysis, needed later, 
which is a generalization of the concepts of \emph{$1$--throws} and \emph{$2$--throws}  considered in \cite {CM2}. 

Suppose that a $(-1)$-curve $C$ lives on a component $P$ in a degeneration 
with two components $P$ and $F$ \
in the central fibre of a family $\mathcal X\to \mathbb D$, 
intersecting  transversely along a curve $R$,
and suppose we are given a line bundle $\mathcal L$ on $\mathcal X$.
Suppose that the intersection number of $C$ 
with the restriction of $\mathcal L$ to $P$ is $-s$. 
Suppose in addition that $C$ meets the double curve $R$ transversely at $m$ points.

For $m=1$ we have the $1$--throw considered in  \cite {CM2},
which reveals that the appropriate matching conditions for a curve on $F$ to be a limit 
is that it must have a point of multiplicity $s$ at the intersection point of $C$ with $R$, 
not simply  having intersection multiplicity $s$ with $R$ at that point.

Now suppose that $m > 1$.
We blow up $C$ in the threefold $\mathcal X$ $m$ times,
thus obtaining a new threefold $\mathcal X'$
and a new family $\mathcal X'\to \mathbb D$.
This blows up $F$ $m$ times at each of the $m$ intersection points of $C$ with $R$,
for a total of $m^2$ blow--ups.
We denote by $\bar F$ the resulting surface.

These blow--ups create $m$ ruled surfaces $Q_{m-1}, Q_{m-2},\ldots, Q_1, Q_0$ 
which are stacked one on the other. 
In the central fiber of $\mathcal X'$, 
$Q_i$ appears with multiplicity $m-i$, for $i=0,\ldots, m-1$.
One checks that $Q_i \cong \mathbb{F}_{i}$, 
with non--positive section $B_i$ and disjoint non--negative section $S_i$; 
on $Q_i$ we have $B_i^2=-i$, $S_i^2=i$, and $S_i \sim B_i + if$, 
where $f$ denotes as usual the fibre class and $\sim$ is the linear equivalence.
$Q_0$ meets the surface $P$ in a section $B_0$ (equal to $C$ on $P$), with $B_0^2=0$.
Each $Q_i$ meets $Q_{i+1}$ so that $S_i$ (on $Q_i$) is identified with $B_{i+1}$ (on $Q_{i+1}$).
Each $Q_i$ also meets the other component $\bar F$ in $m$ fibers of the ruling, corresponding to the $m$ points where $C$ meets $R$.

The normal bundle of $Q_0$ in $\mathcal X'$ is $(-1/m)(B_0+(m-1)S_0 +mf) = -B_0-f$.
For $1< i < m-1$, the normal bundle of $Q_i$ in $\mathcal X'$  is 
$(-1/(m-i)) ((m-i+1)B_i + (m-i-1) S_i + m f)
= -2B_i -(i+1)f$.
For $i=m-1$, the normal bundle of $Q_{m-1}$ in $\mathcal X'$ is $(-1)(2B_{m-1}+mf) = -2B_{m-1} -mf$.

When we pull back the bundle $\mathcal L$ to $\mathcal X'$,
this pull back $\mathcal L'$ restricts to $-sf$ on each $Q_i$.
At this point we make the additional assumption that $s$ is divisible by $m$: write $s = hm$. Twist $\mathcal L'$  by $\mathcal O_{\mathcal X'}(-h(\sum_{i=0}^{m-1} (m-i) Q_i))$.
Let us analyze the restriction of this new bundle on each component of the central fibre. 

First we consider the surface $P$, on which the original curve $C$ sits.  
Since the only exceptional surface that meets $P$ is $Q_0$,
we are twisting the restriction of the bundle on $P$ by $-hmQ_0 = -sQ_0$,
and since $Q_0$ restricts to $C$ on $P$ 
this removes $sC$ from the restriction of the bundle on $P$, 
and then this restriction is trivial on $C$.

The restriction to $Q_0$ is 
$$
-sf - hm{Q_0}|_{Q_0} - h(m-1) {Q_1}|_{Q_0} = -sf - s (-B_0-f) - h(m-1)S_0
= h B_0.
$$
For $1<i<m-1$, the restriction to $Q_i$ is 
\begin{align*}
-sf-h(m-i+1){Q_{i-1}}|_{Q_i} -h(m-i){Q_i}|_{Q_i} -h(m- i-1){Q_{i+1}}|_{Q_i}
&=\\
 -sf -h(m-i+1)B_i -h(m-i)(-2B_i-(i+1)f) -h(m-i-1)S_i&=0.
 \end{align*}
Finally for $i = m-1$ the restriction to $Q_{m-1}$ is 
$$
-sf -2h{Q_{m-2}}|_{Q_{m-1}} - h{Q_{m- 1}}|_{Q_{m-1}}=0.
$$

The above analysis shows that the bundle is now trivial on $Q_{m-1}, Q_{m-2},\ldots, Q_1$, and non--trivial only on $Q_0$, where it consists of $hB_0$, i.e., $h$ horizontal sections. Therefore the matching divisor on $\bar F$ does not meet any of the exceptional divisors of the first $m-1$ blow--ups, and meets only the last ones $h$ times at each of the $m$ points. Moreover, there is a \emph{correspondence} on the divisors on the final exceptional curves, namely they must all agree with $h$ horizontal sections. In other words, any one of these intersections determines all the other $m-1$ ones. This behaviour of the curves on $\bar F$ means that the curve on $F$ must have at each of the $m$ points of the intersection of $C$ and $R$, $m$ infinitely near points of multiplicity $h$ along $R$. We denote this phenomenon by $[h^m]_R$. Hence the matching conditions for the curves on $F$ can be written as $([h^m]_R)^m$, plus the correspondence
coming from agreement with the $h$ horizontal sections.

We can summarize what we proved in this section in the following statement:

\begin{proposition}\label{prop:refined} Suppose we have a semistable degeneration of surfaces $\mathcal X\to \mathbb D$ over a disc $\mathbb D$ and a line bundle $\L$ on $\mathcal X$ which restricts to line bundles on every component of the central fibre. Let $P$ be a component of the central fibre containing a $(-1)$--curve $C$ which is not a double curve and intersects the double curve $R$ transversally at $m$ points $p_1,\ldots, p_m$. Suppose that $C\cdot \mathcal L=-hm$, with $h>0$. Then any curve on the central fibre that is a limit of a curve in the general fibre in the linear system determined by the restriction of $\L$, must satisfy the following conditions:  for every point $p_i$ the curve on the component different from $P$ has the singularity of type $[h^m]_R$ at $p_i$ and the final $h$ infinitely near points to the $p_i$'s of order $m$ correspond in the sense described above. 
\end{proposition}

We note that the correspondence condition in the above proposition
imposes an additional $h(m-1)$ conditions on the divisors on the surface meeting $P$ along $R$.

\section{The proof of the case $m=2$}\label{sec:proof2}

We focus in this section on the case $m=2$. We will consider the degeneration described in \S \ref {sec:deg} and first we want to describe the distribution of the multiplicities and the limit linear systems on $P$ and $F$. For convenience we set $n=d-1$ and note that there are $4n$ simple points in the case $m=2$, which we will distribute evenly among $P$ and $F$. So the limit linear systems will be
$$\L_P:=\L_{kn+t}(k(n-1),k^{2n}),\quad \L_F:=\L_{k(n+1)}(kn+t,k^{2n})$$
with $t$ the twisting parameter. 

In order to prove Theorem \ref {thm:main}(a), for $m=2$, we will use the refined matching approach. This requires that we prove that for any twisting parameter $t$ there is no limit curve satisfying the refined matching conditions stated in Proposition \ref {prop:refined}. 

Consider the curve class (useful on both $P$ and $F$) equal to $\L_n(n-1, 1^{2n})$. We note that this linear system is of dimension $0$ and consists of a unique $(-1)$--curve $C$. 

The linear system $\L_{kn}(k(n-1),k^{2n})$ is equal to $|kC|$, and has  dimension 0. Therefore if $t<0$, then $\L_P$ is empty. Hence we may assume $t\geq 0$. 

Let us analyze $\L_F$. The lines through the first point and through any one of the other $2n$ points split off with multiplicity $t$. The residual system has the form
$\L_{k(n+1)-2nt}(kn+t-2tn, (k-t)^{2n})$. Now we intersect this system with $C$ and get $-t(n-1)$. So $C$ splits   $t(n-1)$ times. The further residual system is  $\L_F'=\L_{(n+1)(k-tn)}(n(k-tn), (k-tn)^{2n})$. For $\L_F'$ to be effective, one needs $t\leq k/n$, which we will assume from now on. A sequence of $n$ quadratic transformations (each based at the first point and at two of the $2n$ points) brings $\L_F'$ to the complete linear system $\L_{k-tn}$. 

As for $\L_P$, one sees that $C$ splits off with multiplicity $k-tn$ and the residual system is $\L_P'=\L_{t(n^2+1)}(tn(n-1), (tn)^{2n})$. A sequence of $n$ quadratic transformations (each based at the first point and at two of the $2n$ points) brings this system to $\L_{t(n+1)}(t^{2n})$. 

Let us see what the refined matching implies on $\L_P'$ or rather on its Cremona transform $\L_{t(n+1)}(t^{2n})$. 

Each of the $2n$ lines splitting $t$ times  from $\L_F$ are $(-1)$--curves meeting the double curve $R$ once. Hence in the notation of the previous section $m=1$ and $s=h=t$ and therefore we are imposing $2n$ points of multiplicity $t$ to the linear system $\L_{t(n+1)}(t^{2n})$. These points are located along the curve $T$, the Cremona image of $R$ on $P$, which is easy to see to be equal to a curve of degree $n+1$, with a point of multiplicity $n$. 

Also the curve $C$ splits $t(n-1)$ times from $\L_F$ and meets the double curve $R$ transversely at $n-1$ points.  In the notation of the previous section $m=n-1$, $s=t(n-1)=tm$ hence $h=t$. Therefore we are imposing to $\L_{t(n+1)}(t^{2n})$ also the multiple points $([t^{n-1}]_T)^{n-1}$, plus the correspondence. 

Eventually the resulting system on $P$ is Cremona equivalent to the system $\L$ of plane curves of degree $t(n+1)$, with $(n+1)^2$ points of multiplicity $t$, plus the correspondence. These $(n+1)^2$ points are distributed in $2n$ general points, $2n$ general points on $T$, and $n-1$ general points of type $[t^{n-1}]_T$. 

Assume now $t>0$. We want to prove that $\L$ is empty and therefore $\L_P$ is empty. To prove this, we need the following:

\begin{lemma}\label{lem:unique} For any $t>0$ the linear system  of plane curves of degree $t(n+1)$ with $n-1$ general points of type $[t^{n-1}]_T$, with $2n$ general points of multiplicity $t$ on $T$ and $2n$ additional general points of multiplicity $t$ consists of  at most one element.
\end{lemma}  

\begin{proof} We specialize the configuration of the imposed multiple points to $n-1$ general points of type $[t^{n-1}]_T$, and  with $4n$ more general points of multiplicity $t$ on $T$. This is a total $(n+1)^2$ points (some of them are infinitely near) forming a divisor $D$ on $T$ supported on the smooth locus of $T$. By generality, for no positive integer $t$, $tD$ belongs to $|\O_T(t(n+1))|$. So any curve of degree $t(n+1)$ with the above multiple points on  $T$ must contain $T$. To the residual curve we may apply the same argument, so $T$ recursively splits; by induction we conclude that the only possible member of the system is $tT$.  This implies the assertion. 
\end{proof}

To prove that $\L$ is empty, we notice that the possible unique curve satisfying the multiplicity conditions imposed on $\L$ (see Lemma \ref {lem:unique}) will not satisfy the required correspondence as soon as $m=n-1\geq 2$, i.e., $n\geq 3$, hence $\L_P$ is empty.

Finally we have to deal with the case $t=0$. In this case we take $k=hn$ and  $\L_P$ consists of the unique curve  $hnC$. Now $C$ is a $(-1)$--curve that intersects $R$ transversely at $n$ points. Therefore, in the notation of \S \ref {sec:nthrow}, we have $m=n$, $s=k$. So the refined matching implies that we eventually have to impose to $\L_F$, or rather to its Cremona transform $\L_{hn}$, $n$ points of type $[h^n]_{T'}$, where $T'$ is  the Cremona image of the double curve $R$, plus the correspondence. Note that $T'$ is  a curve of degree $n$ with a point of multiplicity $n-1$. 

By the same argument as in Lemma \ref {lem:unique}, we see that the only curve verifying all the above conditions is $hT'$, so in the original analysis it is $hR$ plus some exceptional curves which appear in the Cremona transformation. However, as we saw in the refined matching analysis, on $P$ the bundle is now trivial, so the corresponding section on $P$ has to vanish identically on $P$. This shows that there is no limit curve in the case $t=0$ either. 

Eventually we have seen that for any twisting parameter $t$ there is no limiting curve verifying the refined matching conditions, finishing the proof of Theorem \ref {thm:main}(a) for $m=2$ and $d\geq 4$. 

\section{The proof for $m>2$}\label{sec:proof}

In this section we will complete the proof of Theorem \ref {thm:main}(a) in the case $m\geq 3$, arguing by induction on $m$ (the case $m=2$ for all $d\geq 4$  is the starting case of the induction). For this we will again use the degeneration as in \S \ref {sec:deg} and the naive matching approach will be sufficient.

Let us describe the limit linear systems we will use, i.e.,
$$L_P = \L_{k(d-2)}(k(d-m), k^{(m-2)(2d-m-2)}) = k\xi_{d-2,m-2},$$  
$$L_F = \L_{kd}(k(d-2), k^{4d-4}) = k\xi_{d,2}.$$

By the $m=2$ case, $L_F$ is empty, and therefore also the kernel system is empty. Hence it suffices to show that the kernel system on $P$  is also empty.

First consider the case $m=3$. Then, by Theorem \ref{thm:main}(b), $L_P$ is composed with a pencil of rational curves, and the kernel system is empty because it
consists of the members of $L_P$ that vanish along the double curve $R$, which is a general line on $P$. This proves the $m=3$ case for all $d\geq 5$ (remember that $m\leq d-2$).

Next assume $m\geq 4$, and therefore, since $m\leq d-2$, we have $d\geq 6$. Then, by induction, $L_P$ is empty and hence also the kernel system is empty, finishing the proof of Theorem \ref {thm:main}.

\end{document}